\documentclass[12pt]{article}

\usepackage{epsf,amsmath,amssymb,amsfonts,oldlfont,graphics,oldgerm,latexsym,dsfont}
\usepackage{graphics}
\usepackage[dvips]{color}
\usepackage{epsf}
\usepackage{graphicx}
\usepackage{verbatim}

\newcommand{\Aut}{\ensuremath{\operatorname{Aut}}}
\newcommand{\End}{\ensuremath{\operatorname{End}}}

\newcommand{\id}{\ensuremath{\text{\rm id}}}

\newcommand{\proof}{\noindent{\bf Proof.\ }}
\newcommand{\qed}{\hfill $\square$ \bigskip}

\newcommand{\m}[1]{\textcolor{black}{#1}}
\newcommand{\w}[1]{\textcolor{black}{#1}}
\newcommand{\f}[1]{\textcolor{black}{#1}}

\newtheorem{theorem}{Theorem}
\newtheorem{lemma}[theorem]{Lemma}
\newtheorem{corollary}[theorem]{Corollary}

\begin{document}


\title{Endomorphism Breaking in Graphs}


\author{Wilfried Imrich\thanks{The research was supported by the Austrian Science Fund (FWF): project W1230.}  \\
\small Montanuniversit\"at Leoben, A-8700 Leoben, Austria \\
\small\tt imrich@unileoben.ac.at
\and
Rafa{\l}  Kalinowski\thanks{The research was partially supported by the Polish Ministry of Science and Higher
  Education.}\\
\small AGH University of Science and Technology,\\
\small al. Mickiewicza 30, 30-059 Krakow, Poland\\
\small\tt kalinows@agh.edu.pl
\and
Florian Lehner\thanks{The research was supported by the Austrian Science Fund (FWF): project W1230.}\\
\small Institut f\"{u}r Geometrie, Technische Universit\"{a}t Graz \\ \small Kopernikusgasse 24/IV, A-8010 Graz, Austria\\
\small\tt {f.lehner@tugraz.at}
\and
Monika Pil\'sniak\thanks{The research was partially supported by the Polish Ministry of Science and Higher
  Education.}\\
\small AGH University of Science and Technology,\\
\small al. Mickiewicza 30, 30-059 Krakow, Poland\\
\small\tt pilsniak@agh.edu.pl
\date{}
}


\maketitle

\begin{abstract}
We introduce the {\it endomorphism distinguishing number} $D_e(G)$ of a graph $G$
as the least cardinal $d$ such that $G$ has a vertex coloring with $d$ colors
that is only preserved by the trivial endomorphism. This generalizes the notion
of the distinguishing number $D(G)$ of a graph $G$, which is defined for
automorphisms instead of endomorphisms.

As the number of endomorphisms can vastly exceed the number of automorphisms,
the new concept opens challenging problems, several of which are presented here.
 In particular, we investigate relationships between $D_e(G)$ and
the endomorphism motion of a graph $G$\m{, that is, the least possible number of vertices moved by
a nontrivial endomorphism of $G$}.
Moreover, we extend numerous results about the distinguishing number of finite and infinite
graphs  to the endomorphism distinguishing number. This is the
main concern of the paper.

\vspace{0.7cm}

\end{abstract}

\section{Introduction}

Albertson and Collins~\cite{alco-96} introduced the {\it
distinguishing number} $D(G)$ of a graph $G$  as the least cardinal
$d$ such that $G$ has a labeling with $d$ labels that is only
preserved by the trivial automorphism.

This  concept has spawned numerous
papers, mostly on finite graphs.
But countable  infinite graphs  have also been
investigated with respect to the distinguishing number; see \cite{imje-xx},
\cite{smtuwa-xx},   \cite{tu-11}, and \cite{wazh-07}.
For graphs of higher cardinality compare \cite{imkltr-07}.

The aim of this paper is the presentation of fundamental results for
the endomorphism distinguishing number,
and of open problems.
In particular, we extend the Motion Lemma of Russell and Sundaram \cite{rusu-98} to
endomorphisms,  present \w{endomorphism motion conjectures that
generalize the Infinite Motion Conjecture {\m of Tom Tucker \cite{tu-11}} and the Motion Conjecture of
\cite{cuimle-2014}}, prove
the validity of special cases, and support the conjecture\w{s} by examples.

\section{Definitions and Basic Results}

As the distinguishing number has already been defined, let us note that
$D(G)=1$ for all asymmetric graphs. This means that almost all finite
graphs have distinguishing number one, because almost all graphs are asymmetric,
see Erd{\H{o}}s and R{\'e}nyi \cite{erre-63}.
Clearly $D(G)\geq 2$ for all other graphs. Again, it is natural to conjecture
that almost all of them have distinguishing number two. This is supported by
the observations of Conder and Tucker \cite{cotu-2011}.

However, for the complete graph $K_n$, and the complete bipartite graph
$K_{n,n}$ we have
$D(K_n) = n$, and $D(K_{n,n}) = n+1$. Furthermore,
the distinguishing number of the cycle of length 5 is 3,
but cycles $C_n$ of length $n \geq 6$  have distinguishing
number 2.

This compares with a more general result of Klav\v zar, Wong and Zhu \cite{klwozh-06} and of Collins and Trenk
\cite{cotr-06}, which asserts that
$D(G)\leq \Delta(G) +1$, \w{where $\Delta$ denotes the maximum degree of $G$. E}quality holds if and only if $G$ is a $K_n$, $K_{n,n}$ or $C_5$.
\smallskip

Now to the endomorphism distinguishing number. Before defining it, let us
recall that an {\it endomorphism} of a graph $G=(V,E)$ is a mapping $\phi:\,
V\rightarrow V$ such that for every edge $uv\in E$ its image $\phi(u)\phi(v)$ is an
edge, too.
\bigskip

\noindent
{\bf Definition}
\emph{The endomorphism distinguishing number $D_e(G)$ of a graph $G$ is the
least cardinal $d$ such that $G$ has a labeling with $d$ labels that
is  preserved only by the identity endomorphism of $G$.}
\bigskip

\noindent
Let us add that  we also say \emph{colors} instead of \emph{labels}. If a labeling $c$ is not preserved
by an endomorphism $\phi$, we say that $c$ {\it breaks} $\phi$.
\smallskip

Clearly $D(G) \leq D_e(G)$. For graphs $G$ with $\Aut(G)=\End(G)$
equality holds. Such graphs
are called {\it core graphs}.
Notice that complete graphs and odd cycles
are core graphs, see \cite{goro-01}. Hence
 $D_e(K_n)=n$, $D_e(C_5)=3$, and $D_e(C_{2k+1})=2$ for $k\geq 3$.

Interestingly,  almost all graphs are core graphs, as shown by
Koubek and R{\"o}dl \cite{koro-84}. Because almost all graphs are asymmetric,
this implies that almost all graphs have trivial endomorphism monoid,
that is, $\End(G)=\{\id\}$.  Graphs with trivial endomorphism monoid are called {\it rigid}.
Clearly $D_e(G)=1$ for any rigid graph $G$, and thus
$D_e(G)=1$ for almost all graphs $G$.

$D_e(G)$ can be equal to $D(G)$ even when $\Aut(G)
\subsetneq \End(G)$. For example,  this is
the case for even cycles. We formulate this as a lemma.

\begin{lemma} The automorphism group of even cycles is properly contained in their endomorphism monoid, but
  $D(C_{2k}) = D_e(C_{2k})$ for all $k \geq 2$.
\end{lemma}

\proof  It is easily seen that every even cycle admits proper endomorphisms, that is, endomorphisms that are not
automorphisms. Furthermore, it is readily verified that \m{$|\End(C_4)|=14$ and} $D(C_4) = D_e(C_4)=3$.

Hence, let $k \geq 3$.  Color
the vertices  $v_1,v_2$ and $v_4$ black and all other vertices white, see
Figure \ref{fig:hex}. We wish to show that this coloring is endomorphism
distinguishing. Clearly this coloring  distinguishes all automorphisms.

\begin{figure}[h]
\begin{center}
\includegraphics [width=4cm]{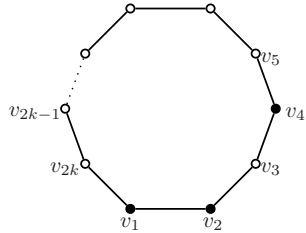}
\caption{Distinguishing an even cycle}
\label{fig:hex}
\end{center}
\end{figure}
\noindent

Let $\phi$ be a
proper endomorphism. It has to map the cycle into a proper connected subgraph of itself. Thus,
$\phi(C_{2k})$ must be a path, say $P$.

Furthermore,  \w{the color of the endpoints of an edge must be preserved under $\phi$.}
Hence $v_1v_2$ is mapped into itself. Because $v_{2k-1}v_{2k}$ is the
only edge with two white endpoints that is adjacent to $v_1v_2$, it must
also be mapped into itself. This fixes $v_{2k-1}, v_{2k}, v_1$ and $v_2$.
But then $v_3$ and $v_4$ are also fixed.

Now we observe that the path $v_4v_5 \cdots v_{2k}v_1$ in $C_{2k}$
has only white interior vertices and that it it has to be mapped into a walk
in $P$ from $v_4$ to $v_1$ that contains only white interior vertices.
Clearly this is not possible. \qed

To show that $D(G)$ can be smaller than $ D_e(G)$, we consider graphs $G$ with
trivial automorphism group but nontrivial
endomorphisms monoid. For such graphs $D(G) = 1$, but $D_e(G) > 1$.
Easy examples are asymmetric, nontrivial trees $T$. For,
every such tree has at least 7 vertices and at least three vertices of degree
1. Let $a$ be a vertex of degree 1 and $b$ its neighbor. Because
$T$ has at least 7 vertices and since it is connected, there must be
a neighbor $c$ of $b$ that is different from $a$. Then the mapping

\[ \phi: v \mapsto \left\{ \begin{array}{l}c \mbox{ if } v = a\\
v \mbox{ otherwise} \end{array}\right. \]
is a nontrivial endomorphism.

\section{The Endomorphism Motion Lemma}

Russel and Sundaram  \cite{rusu-98} proved that the distinguishing number of a graph is
small when every automorphism of $G$ moves many elements.
We generalize this result to endomorphisms and begin with the definition
of motion.

The {\it motion} $m(\phi)$ of a nontrivial endomorphism  $\phi$ of
a graph $G$, is the number of elements it moves:
$$m(\phi)  =  |\{v \in V(G)\,|\, \phi(v) \neq v\}|.$$
The {\it endomorphism motion} of a graph $G$ is
$$m_e(G) = \min_{\phi \in \End(G)\setminus\{\id\}} m(\phi)$$
For example, $m_e(C_4) = 1, m_e(C_5) = 4, m_e(C_{100}) = 49, m_e(K_{100}) = 2$.

In the sequel we will prove  the following generalization of Theorem 1 of Russell
and Sundaram \cite{rusu-98}.

\begin{lemma}[Endomorphism Motion Lemma]
\label{EML} 
For any graph $G$,
\begin{equation}
\label{le:endmo}
d^{\frac{m_e(G)}{2}} \geq  |\End(G)|\,
\end{equation} implies $D_e(G)\leq d.$
\end{lemma}

\w{The proof will be an easy consequence of Lemma \ref{orbnorm},  the Orbit Norm Lemma. We first  define orbits of endomorphisms.}
\bigskip

\noindent{\bf Definition} \emph{
An orbit of an endomorphism $\phi$ of a graph $G$ is an equivalence class with
respect to the equivalence relation $\sim$ on $V(G)$, where $u\sim
v$ if there exist nonnegative integers $i$ and $j$ such that
$\phi^{i}(u)=\phi^{j}(v)$.}
\bigskip

\noindent
The orbits form a partition $V(G) = I_1\cup I_2\cup \cdots \cup I_k$,
$I_i \cap I_j = \emptyset$ for $1\leq i < j\leq k$,
of $V(G)$. For finite graphs it can  be characterized  as the unique
partition with the maximal number of sets that are invariant under \m{the preimage}
$\phi^{-1}$. For infinite graphs we characterize it as the finest
partition that is invariant under $\phi^{-1}$.
For automorphisms it coincides with the cycle decomposition.

The {\it orbit norm} of an endomorphism $\phi$ with the orbits $I_1,
I_2, \ldots, I_k$ is
$$o(\phi) = \sum_{i=1}^k(|I_i| -1)\,,$$
and the {\it endomorphism orbit norm} of a graph $G$ is
$$o(G)  = \min_{\phi \in \End(G)\setminus\{\id\}} o(\phi)\,.$$
Notice that $\phi$ may not move all elements of a nontrivial orbit,
whereas automorphisms move all elements in a nontrivial cycle of the
cycle decomposition. To see this, consider an orbit $I = \{a,b\}$,
where $\phi(a) = b$, and $\phi(b) = b$. Only one element of the
orbit is moved, and the contribution of $I$ to the orbit norm of
$\phi$ is $1$.
Clearly $o(\phi) \geq m(\phi)/2$, and thus $o(G)\geq m_e(G)/2$.

\begin{lemma}[Orbit Norm Lemma]
\label{orbnorm}
A graph $G$ is endomorphism $d$-distinguishable if
$$\sum_{\phi\in\End(G)\setminus\{\id\}}d^{-o(\phi)}<1.$$
\end{lemma}

\proof  We study the behavior of a random $d$-coloring $c$ of G, the probability distribution
given by selecting the color of each vertex independently and uniformly in the set
$\{1, \ldots, d\}$. Fix an endomorphism $\phi \neq \id$ and consider the event
that the random coloring $c$ is  preserved by $\phi$, that is, $c(v)=c(\phi(v))$ for each vertex $v$ of $G$.
Then it is easily seen that
$$\mbox{Prob}[\forall v:\, c(v) = c(\phi(v))] = \left(\frac{1}{d}\right)^{o(\phi)} \leq \left(\frac{1}{d}\right)^{o(G)}\!\!.$$
Collecting together \m{these events, we have
$$\mbox{Prob}[\exists \phi \neq \id\;\forall v:\, c(v) = c(\phi(v))] \leq
\sum_{\phi\in\End(G)\setminus\{\id\}} \left(\frac{1}{d}\right)^{o(\phi)}\!\!.$$
If this sum is strictly less than one, then} there exists a coloring $c$ such that  for all nontrivial $\phi$ there is a $v$, such that
$c(v) \neq c(\phi(v))$, as desired.\qed

\w{
\noindent{\bf Proof of Lemma \ref{EML}.} From
$o(G)\geq m_e(G)/2$ we infer that
\begin{equation}
\label{eq:end}\sum_{\phi\in\End(G)\setminus\{\id\}} d^{-o(\phi)}\,\leq\, (|\End(G)|-1)\,d^{-o(G)}\,\leq\,(|\End(G)|-1)\,d^{-\frac{m_e(G)}{2}}.
\end{equation}
Hence, if $$d^{\frac{m_e(G)}{2}} \geq  |\End(G)|,$$
then the right side of Equation \ref{eq:end} is strictly less than 1, and therefore  also less than $\sum_{\phi\in\End(G)\setminus\{\id\}} d^{-o(\phi)}$.
Now an application of the Orbit Norm Lemma shows that $G$ is $d$-distinguishable. \qed.}

Lemma \ref{EML}  \m{is similar to the {\it Motion Lemma} of Russell and Sundaram \cite[Theorem 1]{rusu-98}, which} asserts
that $G$ is 2-distinguishable if  $$m(G) > 2\log_2 |\mbox{Aut}(G)|\,,$$ where
$$m(G) = \min_{\phi \in \Aut(G)\setminus\{\id\}} m(\phi)\,.$$
\m{Actually, a short look at the proof of  Russell and Sundaram shows that  $G$ is $d$-distinguishable under the weaker assumption}
\begin{equation}
\label{eq:rusu}
m(G) \geq 2\log_d |\mbox{Aut}(G)|.
\end{equation} Thus, our Endomorphism
Motion Lemma is a direct generalization of the
Motion Lemma of Russell and Sundaram.

The Motion Lemma allows \w{the computation of}
the distinguishing number of many classes of finite graphs.
We know of no such applications for the Endomorphism Motion Lemma, but will show
the applicability of its generalization to infinite graphs.

\section{ Infinite graphs}\label{main}

Suppose we are given an infinite graph $G$ with infinite endomorphism motion
$m_e(G)$ and wish to generalize Equation \ref{le:endmo} to this case for finite $d$.
Notice that
$$d^{m_e(G)/2} = d^{m_e(G)} = 2^{m_e(G)}$$
in this situation. Thus the natural generalization would be that
\begin{equation}
\label{eq:infendmo}
2^{m_e(G)} \geq |\End(G)|
\end{equation} implies endomorphism 2-distinguishabilty.
We formulate this as a conjecture.
\bigskip

{\bf Endomorphism Motion Conjecture} \emph{Let $G$ be a
connected, infinite graph with endomorphism motion $m_e(G)$. If $2^{m_e(G)} \geq
|\!\End(G)|$, then $D_e(G) = 2$.} \medskip

\w{This is a generalization of the Motion Conjecture of \cite{cuimle-2014} for automorphisms of graphs. Notice that we
assume connectedness now, which we did not do before. The reason is, that we
not only have to break all endomorphisms of every connected component if the graph is disconnected, but that we also have to worry about breaking mappings between possibly infinitely many different connected components, which  requires  extra effort.}

\w{A special case are countable graphs. Let $G$ be an infinite, connected countable graph with infinite endomorphism motion $m_e(G)$. Then}
$m_e(G) = \aleph_0$ and $2^{m_e(G)} = 2^{\aleph_0} = \textswab{c}$,
where $\textswab{c}$ denotes \w{t}he cardinality of the continuum.

\w{Notice, for countable graphs,}  $|\End(G)| \leq \aleph_0^{\aleph_0}
= 2^{\aleph_0} = \textswab{c}$. This means that Equation \ref{eq:infendmo} is always satisfied
for countably infinite graphs with infinite motion. This motivates the following conjecture:
\bigskip

\w{{\bf Endomorphism Motion Conjecture for Countable Graphs}}
 \emph{Let $G$ be a
countable connected graph with infinite endomorphism motion. Then $G$ is endomorphism 2-distin\-guishable.} \medskip

In the last section we will verify this conjecture for countable trees with infinite endomorphism motion. Their  endomorphism monoids are uncountable and we will \w{show} that they have endomorphism distinguishing number
 2.

We now \w{prove} the \w{conjecture for countable endomorphism monoids. In fact, we  show that almost every coloring is distinguishing if the endomorphism monoid is countable.}

\begin{theorem}\label{TTconj_w}
Let $G$ be a graph with infinite \w{endomorphism motion} whose endomorphism monoid is countable. Let $c$ be a random $2$-coloring where all vertices have been colored independently and assume that there is an $\varepsilon > 0$ such that, for every vertex $v$, the probability that it is assigned \w{a} color $x \in \{\text{black}, \text{white}\}$ satisfies
\[
	\varepsilon \leq {\operatorname{Prob}}\left [c(v) = x \right ] \leq 1-\varepsilon.
\]
Then $c$ is almost surely distinguishing.
\end{theorem}

\begin{proof}
First, \m {let $\phi$ be a fixed, non-trivial} endomorphism of $G$. Since the motion of $\phi$ is infinite we can find infinitely many disjoint pairs $\{v_i, \phi (v_i) \}$. Clearly the colorings of these pairs are independent and the probability that $\phi$ preserves the coloring in  any of the pairs is bounded from above by some constant $\varepsilon' < 1$. Now
\[
{\operatorname{Prob}}\left [\phi \text{ preserves }c \right ]
\leq {\operatorname{Prob}} \left [ \forall i \colon c(v_i) = c(\phi (v_i)) \right ]
= 0.
\]

Since there are only countably many endomorphisms we can use $\sigma$-subadditivity of the probability measure to conclude that \m{
\[
{\operatorname{Prob}} \left[ \exists \phi \in \End (G)\setminus \{id\}\colon \phi \text{ preserves c} \right]
\leq \sum_{\phi \in \End (G)\setminus \{\id\}} {\operatorname{Prob}}\left [\phi \text{ preserves }c \right ]
= 0,
\]}
which completes the proof.
\qed
\end{proof}

We will usually only use the following Corollary of Theorem \ref{TTconj_w}.

\begin{corollary} \label{countmon}
Let $G$ be a graph with infinite motion whose endomorphism monoid is countable. Then $$D_e(G) =2.$$
\end{corollary}

\w{The endomorphism motion conjecture for countable
graphs generalises the}
\bigskip

{\bf  Infinite Motion Conjecture of Tucker \cite{tu-11}} \emph{Let $G$ be a connected, locally finite
infinite graph with infinite motion. Then $G$ is 2-distinguishable.}\medskip

\w{It was shown in \cite{cuimle-2014}, and follows from Theorem \ref{TTconj_w}, that it is true for countable $\mbox{Aut}(G)$.}
 There are numerous applications of this result, see \cite{imsm-xx}.

For the  Endomorphism Motion Conjecture for Countable Graphs we have the following \w{generalization of
 \cite[Theorem 3.2]{imkltr-07}:}

\begin{theorem}
\label{thm:fggroups}
Let $\Gamma$ be a finitely generated infinite group. Then there is a 2-coloring of the elements of\, $\Gamma$, such that the
identity endomorphism of\, $\Gamma$ is the only endomorphism that preserves this coloring. In other words,
finitely generated groups are endomorphism 2-distinguishable.
\end{theorem}

\proof Let $S$ be a finite set of generators of $\Gamma$ that is closed under inversion.
Since every element $g$ of $\Gamma$ can be represented as a product
$s_1s_2\cdots s_k$
of finite length in elements of $S$, we infer that $\Gamma$ is countable.

Also, if $\phi \in \mbox{End}(\Gamma)$, then
$$\phi(g) = \phi(s_1s_2\cdots s_k) = \phi(s_1)\phi(s_2) \cdots \phi(s_k).$$
Hence, every endomorphism $\phi$ is determined by the finite set
$$\phi(S)=\{\phi(s)\, | \,s \in S\}.$$
Because every $\phi(s)$ is a word of finite length in elements of $S$ there are only countably many
elements in $\phi(S)$. Hence $\mbox{End}(\Gamma)$ is
countable.

Now, let us consider the motion of the nonidentity elements of $\mbox{End}(\Gamma)$.
Let $\phi$ be such an element and consider the set
$$\mbox{Fix}(\phi)= \{g\in  \Gamma \,|\, \phi(g) = g\}.$$
It is easily seen that these elements form a subgroup of $\Gamma$. Since $\phi$ does not fix all
elements of $\Gamma$ it is a proper subgroup. Since its smallest index is two,
the set $\Gamma \setminus \mbox{Fix}(\phi)$ is infinite. Thus $m(\phi)$ is infinite. As  $\phi$ was arbitrarily
chosen, $\Gamma$ has infinite endomorphism motion.

By Corollary \ref{countmon} we conclude that $\Gamma$ is 2-distinguishable. \qed

The next theorem shows that the endomorphism motion conjecture is true if
$m_e(G) =|\!\End(G)|$, even if $m_e(G)$ is not countable.

\begin{theorem}
\label{thm:special}
Let $G$ be a connected graph with uncountable endomorphism motion. Then $|\!\End(G)| \leq m_e(G)$ implies $D_e(G)=2$.
\end{theorem}

\begin{proof}
Set $\textswab{n} = |\!\End(G)|$, and let $\zeta$ be the smallest ordinal number whose underlying set has cardinality $\textswab{n}$.
Furthermore, choose a well ordering $\prec$ of $A = \End(G)\setminus \{\id\}$ of order type $\zeta$, and let  $\phi_0$ be the smallest element
with respect to $\prec$. Then the cardinality of the set of all elements of $A$ between $\phi_0$ and any other $\phi \in A$ is smaller than $\textswab{n} \leq m_e(G)$.

Now we color all vertices of $G$ white and use transfinite induction to break all endomorphisms by coloring selected vertices black.
 By the assumptions of the theorem, there exists a vertex $v_0$  that is not fixed by $\phi_0$.  We color it black. This coloring breaks
$\phi_0$.

For the induction step, let $\psi \in A$. Suppose we have already broken all $\phi \prec \psi$ by pairs of vertices $(v_\phi, \phi(v_\phi))$,  where $v_\phi$ and $\phi(v_\phi)$ have distinct colors. Clearly, the cardinality of the set $R$ of all $(v_\phi, \phi(v_\phi))$, $\phi \prec \psi$, is less than $ \textswab{n} \geq m_e(G)$. By assumption, $\psi$ moves at least $m_e(G)$ vertices. Since there are still $\textswab{n}$ vertices not in $R$, there must be a  vertex $v_{\psi}$ that does not meet $R$. If $\psi(v_\psi)$ is white, we color $v_\psi$ black. This coloring breaks $\psi$. \qed
\end{proof}

\begin{corollary}
\label{cor:special}
Let $G$ be a connected graph with uncountable endomorphism motion. If the general continuum hypothesis holds, and if
$ |\!\End(G)| < 2^{m_e(G)}\!,$\, then $D_e(G) =2$.
\end {corollary}

\proof By the generalized continuum
hypothesis $2^{m_e(G)}$ is the successor of $m_e(G)$. Hence, the inequality $2^{m_e(G)}
>|\!\End(G)|$ is equivalent to ${m_e(G)} \geq |\!\End(G)|$. \qed

\section{ Examples and outlook }\label{ex}

So far we have only determined the endomorphism distinguishing numbers of
core graphs, such as the complete graph and odd cycles, and proved that
$D_e(C_{2k}) = 2$ for $k \geq 3$. Furthermore, it is easily seen that
$D_e(K_{n,n}) = n+1$ and $D_e(K_{m,n}) = \max(m,n)$ if $m \neq n$.

In the case of infinite structures we proved Theorem \ref{thm:fggroups}, which shows
that $D_e(\Gamma) = 2$ for finitely generated, infinite groups $\Gamma$.

We will now determine the endomorphism distinguishing numbers of finite and
infinite paths and we begin with the following lemma.

\begin{lemma}\label{dist2}
Let $\phi$ be an endomorphism of a (possibly infinite) tree $G$ such that
$\phi(u)=\phi(v)$ for two distinct vertices $u,v$.  Then there exist
two vertices $x,y$ on the path between $u$ and $v$ such that
$\phi(x)=\phi(y)$ and ${\rm dist}(x,y)=2$.
\end{lemma}

\begin{proof}
 Suppose dist$(u,v)\neq 2$.
Hence dist$(u,v)>2$. Let $P$ be the path  connecting $u$ and
$v$ in $G$, and let $P'$ be the subgraph induced by the image $\phi(P)$.
Clearly, $P'$ is a finite tree with at least one edge.

Because every nontrivial finite tree
has at least two pendant vertices, there must be a pendant
vertex $w$ of $P'$ that is different from $\phi(u) = \phi(v)$.
Thus $w=\phi(z)$ for some internal vertex $z$ of
$P$. If $x$ and $y$ are the two neighbors of $z$ on $P$, then
clearly $\phi(x)=\phi(y)$ and dist$(x,y)=2$. \qed
\end{proof}

The above lemma implies  the following corollary for
finite graphs, because any injective endomorphism of a finite graph is
an automorphism.

\begin{corollary}\label{dist2fin}
Let $G$ be a finite tree. Then for every
$\phi\in$\End$(G)\setminus$\Aut$(G)$ there exist two vertices $x,y$
of distance 2 such that $\phi(x)=\phi(y)$. \qed
\end{corollary}

\begin{lemma}\label{le:path}
The endomorphism distinguishing number of
all finite paths $P_n$ of order $n\geq 2$ is two.
\end{lemma}

\begin{proof}
Clearly, $D_e(P_n)\geq
2$ since $\End(P_n)\neq\Aut(P_n)$.
To see that $D_e(P_n)=2$ consider the following labeling
$$c(P_n)=\left\{ \begin{array}{ll} (11221122.....1122)& \mbox{ if } n\equiv 0\mod 4 \\
(11221122... 11221)& \mbox{ if } n\equiv 1 \mod 4\\
(1221122.....22112) & \mbox{ if } n\equiv 2 \mod 4 \\
(11221122... 22112) &\mbox{ if } n\equiv 3 \mod 4
\end{array}\right. .$$
The only nontrivial automorphism of a path (symmetry with respect to
the center) does not preserve this  labeling. By Corollary
\ref{dist2fin}, any  other endomomorphism
$\phi\in\End(G)\setminus\Aut(G)$ has to identify two vertices of distance
two. Then $\phi$ cannot preserve the coloring, because any two vertices of
distance two  have distinct labels.
\qed\end{proof}

 Next let us consider the ray and the double ray which can be viewed as an infinite \w{analogs} to finite paths. It turns out that their endomorphism distinguishing number is $2$ as well.

\begin{lemma}\label{le:doubleray}
The endomorphism distinguishing number of the infinite ray and of the infinite double ray is two.
\end{lemma}

Later in this section  Theorem \ref{thm:inftrees} will show that every countable tree with at most one pendant vertex has endomorphism distinguishing number two. Clearly Lemma \ref{le:doubleray} constitutes a special case of this result. It is also worth noting that by the following theorem every double ray has infinite endomorphism motion. Hence we verify the Endomorphism Motion Conjecture for the class of countable trees.

\begin{theorem}\label{thm:treemotion}
A\w{n infinite} tree has infinite endomorphism motion if and only if it has no pendant vertices.
\end{theorem}

The proof uses the following lemma which may be interesting \m{as such}. Note that in the statement of the lemma there is no restriction  on the cardinality of the tree or the motion of the endomorphism.

\begin{lemma}\label{le:fpconnected}
Let $T$ be a tree and let $\phi$ be an endomorphism of $T$. Then the set of fixed points of $\phi$ induces a connected subgraph of $T$.
\end{lemma}

\begin{proof}
Denote by Fix$(\phi)$ the set of fixed points of $\phi$ and assume that it does not induce a connected subgraph. Consider two vertices $v_1, v_2 \in \mbox{Fix}(\phi) $ lying in different components of this graph.

Then $\phi$ maps the unique path in $T$ from $v_1$ to $v_2$ to a $v_1$-$v_2$-walk of the same length. But the only such walk is the path connecting $v_1$ and $v_2$, so this path has to be fixed pointwise. \qed
\end{proof}

{\bf Proof of Theorem \ref{thm:treemotion}}.
Clearly, if \m{an infinite tree} has a pendant vertex,  then there is an endomorphism which moves only this vertex and fixes everything else.

So let $T = (V,E)$ be a tree without pendant vertices and let $\phi$ be a nontrivial endomorphism of $T$. Assume that the motion of $\phi$ \w{is} finite. Then the set $\mbox{Fix}(\phi)$ of fixed points of $\phi$ contains all but finitely many vertices of $T$. Since $T$ has no pendant vertices such a set does not induce a connected subgraph. This contradicts Lemma \ref{le:fpconnected}.
\qed

Now that we have characterised the trees with infinite endomorphism motion, we would like to show that all of them have endomorphism distinguishing number 2.

\begin{theorem}\label{thm:inftrees}
The endomorphism distinguishing number of countable trees $T$ with at most one  pendant vertex
is 2.
\end{theorem}

\begin{proof}
The proof consists of two stages. First we color part of the vertices such that every endomorphism which preserves this partial coloring has to fix all distances from a given vertex $v_0$. Then we color the other vertices in order to break all remaining endomorphisms.

For the first part of the proof,  let $v_0$ be a pendant vertex of $T$, or any vertex if $T$ is a tree without pendant vertices. Denote by $S_n$ the set of vertices at distance $n$ from $v_0$, that is the sphere of radius $n$ with center $v_0$. Now color $v_0$ white and all of $S_1$ and $S_2$ black. Periodically color all subsequent spheres according to the pattern outlined in Figure \ref{fig:coloring}. In other words always color two spheres white, then four spheres black, leave \f{two spheres} uncolored, color another four spheres black and proceed inductively. Furthermore, we require that \f{adjacent uncolored vertices are assigned different colors in the second step of the proof}.

\begin{figure}[h]
\begin{center}
\includegraphics{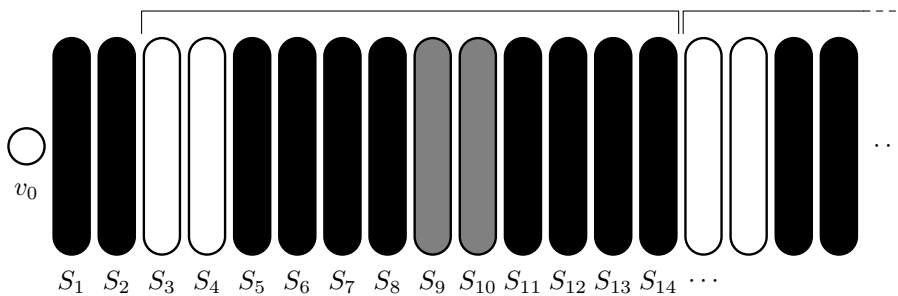}
\caption{Coloring of the spheres in the first part of the proof of Theorem \ref{thm:inftrees} with the period of the periodic part indicated at the top. Grey spheres are left uncolored for the second stage of the proof.}
\label{fig:coloring}
\end{center}
\end{figure}

Now we claim that this coloring fixes $v_0$ in every endomorphism. To prove this consider a ray $v_0v_1v_2v_3\ldots$ starting at $v_0$. Clearly $v_i \in S_i$ holds for every $i$. Assume that there is a color preserving endomorphism $\phi$ of $T$ which does not fix $v_0$ and consider the image of the previously chosen ray under $\phi$, that is, let $\tilde v_i = \phi (v_i)$. Clearly $\tilde v_0$ has to lie either in a white sphere or in a sphere which has not yet been colored. We will look at those cases and show that all of them lead to a contradiction. So assume that $\tilde v_0 \in S_k$ for some $k > 0$.
\begin{itemize}
\item If $k=3$, then $\tilde v_1$ must lie in $S_2$ since it must be a black neighbor of $\tilde v_0$. For similar reasons $\tilde v_2 \in S_1$ and $\tilde v_3 = v_0$ must hold. Now $\tilde v_4$ has to be a white neighbor of $\tilde v_3$ but $v_0$ only has black neighbors, a contradiction.
\item If $k \in 3 + 12 {\mathbb N}$ we get $\tilde v_1  \in S_{k-1}$ and $\tilde v_2 \in S_{k-2}$ by the same argument as above. Now $\tilde v_3$ would need to be a white neighbor of $\tilde v_2$ but $\tilde v_2$ only has black neighbors.
\item If $k \in 4 + 12 {\mathbb N} _0$, \m{where ${\mathbb N} _0={\mathbb N}\cup\{0\}$}, then, for similar reasons as in the previous cases, $\tilde v_1 \in S_{k+1}$ and $\tilde v_2 \in S_{k+2}$. Again $\tilde v_2$ has no white neighbors.
\item If $k \in 9 + 12 {\mathbb N} _0$, \f{then $\tilde v_2$ lies in one of $S_{k-2}$, $S_k$ and $S_{k+2}$. In the first case $\tilde v_2$ clearly has no white neighbors. In the other cases it may have a white neighbor $\tilde v_3$ in $S_{k+1}$, but then $\tilde v_3$ has no white neighbors, because its neighbor in $S_{k}$ must have a different color.}
\item If $k \in 10 + 12 {\mathbb N} _0$,  \f{we can use an argument that is symmetric to the previous case}.
\end{itemize}
Since there are no more cases left we can conclude that $v_0$ has to be fixed by every endomorphism which preserves this coloring.

However, \m{we wish to prove that} such endomorphism $\phi$ preserves all distances from $v_0$, that is, \w{that} $\phi$ maps $S_k$ into itself for each $k$.

\f{We first show that any $v \in S_k$ for $k \in 2 + 12 {\mathbb N}_0$ must have its image in $S_l$ for some $l \in 2 + 12 {\mathbb N}_0$. Since $v_0$ is fixed, $v$ must be mapped to a vertex at even distance from $v_0$. Furthermore, this vertex must be black and have a white neighbor, which again must have a white neighbor. It is easy to check that the only vertices for which all of this holds lie in $S_l$ for some $l \in 2 + 12 {\mathbb N}_0$.}

\f{Now assume that $\phi$ does not map $S_k$ into itself for every $k$} and consider the smallest $k$ such that $\phi (S_k) \nsubseteq S_k$. Then there must be some vertex $u \in S_k$ such that $\phi (u) \in S_{k-2}$. This immediately implies that $k \notin \{ 1, 2 \}$ and that $k \notin \{3,4,5,6\} + 12 {\mathbb N}_0$, because otherwise a white vertex would be mapped to a black vertex or vice versa. In order to treat the remaining cases,  \f{consider a vertex $v \in S_l$ whose predecessor in $S_k$ is $u$, where $l$ is chosen to be minimal with respect to the properties $l > k$, $l \in 2 + 12 {\mathbb N}_0$. The unique $u$-$v$-path in $T$ must be mapped to a $\phi(u)$-$\phi(v)$-walk with length at most $l-k$. This implies that $\phi(v)$ cannot lie in $S_{m}$ for $m \geq l$. The $u$-$v$-path does not contain two consecutive white vertices, hence the $\phi(u)$-$\phi(v)$-walk cannot cross the two consecutive white layers $S_{l-10}$ and $S_{l-11}$. So $\phi(v)$ cannot lie in $S_{m}$ for $m \leq l-12$. But this contradicts the fact that $\phi(v)$ must lie in some $S_{m}$ for $m \in 2 + 12 {\mathbb N} _0$.}

This completes the proof of the fact that all distances from $v_0$ are fixed by any endomorphism which preserves such a coloring.

\w{For the second part of the proof, consider any enumeration $(v_i)_{i \geq 0}$  of the vertices of $T$ such that, for all
$i \geq 0$, we have $v_i \in S_j$ for some $j < 12i+9$. It is easy to see that
such an enumeration is possible.}
Now color all vertices in $S_{12i+9}$ whose predecessor is $v_i$ black and color all other vertices in this sphere white. \f{Color the vertices of $S_{12i+10}$ whose predecessor is $v_i$ white, and color all other vertices in this sphere black.}

We claim that the so obtained coloring is not preserved by any  endomorphism but the identity. We already know that every color preserving endomorphism $\phi$ maps every sphere $S_k$ into itself. Assume that there is a vertex $v_i$ which is not fixed by $\phi$. Then it is easy to see that all vertices in $S_{12i+9}$ whose predecessor is $v_i$ will be mapped to vertices whose predecessor is $\phi (v_i)$. Hence $\phi $ is not color preserving.
\qed
\end{proof}

\m{We conjecture that this result can be extended to
 uncountable trees.} One does need a lower bound
on the minimum degree though, see \cite{imkltr-07}.
As we already noted, the fact that $D_e(T) = 2$, together with
the observations that  $|\mbox{End}(T)| = \textswab{c}$ and
$m_{e}(T) = \aleph_0$, supports
 the Endomorphism Motion Conjecture.
Of course, a proof of the Endomorphism Motion Conjecture is still not in sight,
not even for countable structures.

Finally, the computation of $D_e(Q_k)$ seems to be an interesting
problem, even for finite  cubes. Similarly,  the computation of $D_e(K_n^k)$,
where $K_n^k$ denotes the $k$-th Cartesian
power\footnote{For the definition of the Cartesian product and
Cartesian powers see \cite{haimkl-11}.} of $K_n$, looks demanding.

\w{\noindent{\bf Acknowledgement} We wish to thank one of the referees for numerous clarifying remarks and helpful comments.}


\end{document}